%% file: arxiv.tex
\documentclass[11pt]{article}
\usepackage{amsmath}
\usepackage{amssymb}
\usepackage{latexsym}
\usepackage{amsthm}
\usepackage{mathrsfs} 
\usepackage{wasysym}
\usepackage{fancyhdr}
\usepackage{xcolor}
\usepackage{amsfonts}
\usepackage{amssymb}
\usepackage{epsfig}
\usepackage{epstopdf}
\usepackage[normalem]{ulem}
\usepackage[T1]{fontenc}
\usepackage{eucal}
\usepackage{tikz}
\usepackage[utf8]{inputenc}
\usepackage{hyperref}
\usepackage{algorithm}
\usepackage{algpseudocode}
\usepackage{subcaption}

\newtheorem{Theorem}{Theorem}[section]

\newtheorem{Lemma}[Theorem]{Lemma}

\title{A note on the injective edge coloring of graphs with bounded maximum degree}
\author{Lin Tian\thanks{Department of Mathematical Sciences, Middle Tennessee State University, Murfreesboro, TN 37132, USA. Email: \texttt{lt5b@mtmail.mtsu.edu}}
\and Runze Wang\thanks{Corresponding author. Department of Mathematics and Computer Science, Augustana College, Rock Island, IL 61201, USA. Email: \texttt{runze.w@hotmail.com}}}
\date{}

\begin{document}
	\sloppy
	\maketitle
	\begin{abstract}
		Let $G$ be a graph. Under an \emph{injective edge coloring} of $G$, any two edges at distance $2$ or belonging to a common triangle receive distinct colors. The \emph{injective chromatic index} of $G$, denoted by $\chi'_{inj}(G)$, is the smallest integer $k$ such that $G$ admits an injective edge coloring with $k$ colors. Let $\Delta$ be the maximum degree of $G$. Ferdjallah et al.~proved that $\chi'_{inj}(G) \le 2(\Delta-1)^2$. In this paper, we improve this bound by proving that $\chi'_{inj}(G) \le 2(\Delta - 1)^2 -\Delta + 3$ if $\Delta \ge 4$.
	\end{abstract}
	
	\section{Introduction}
	All graphs considered in this paper are finite and simple, without loops or multiple edges. Let $G$ be a graph with vertex set $V(G)$ and edge set $E(G)$. Denote by $\Delta(G)$ (or simply $\Delta$) the maximum degree of $G$, denote by $\delta(G)$ the minimum degree of $G$, and denote by $\omega(G)$ the clique number of $G$. For $v\in V(G)$, denote by $\deg_G(v)$ the degree of $v$, and denote by $N_G(v)$ the neighborhood of $v$. For $X \subseteq V(G)$, let $G[X]$ be the subgraph induced by $X$, and let $e_G(X) = |E(G[X])|$.  For two edges $e_1, e_2\in E(G)$, the distance between $e_1$ and $e_2$ is the distance between the corresponding vertices in the line graph of $G$. (Note that the distance between a pair of incident edges is~$1$.) For discrete intervals, let $[n]=\{1,2,\dots,n\}$ for any positive integer $n$.
	
	The notion of injective edge coloring was introduced by Cardoso et al.~\cite{cardoso-2015}. For a graph $G$, an \emph{injective edge coloring} of $G$ is an edge coloring such that any two edges at distance $2$ or belonging to a common triangle receive distinct colors. The \emph{injective chromatic index} of $G$, denoted by $\chi'_{inj}(G)$, is the smallest integer $k$ such that $G$ admits an injective edge coloring with $k$ colors. Ferdjallah et al.~\cite{FBSA-2022} proved that if $\Delta \ge 3$, then
	\[\chi_{inj}'(G) \le 2(\Delta - 1)^2.\]
	
	Later, the upper bound for $\Delta = 3$ was improved by Kostochka et al.~\cite{Kostochka}. They proved that $\chi_{inj}'(G) \le 7$ if $\Delta = 3$. However, for $\Delta \ge 4$, although some better bounds have been proven for sparse graphs (see e.g.~\cite{BWZZ,FL,HZ,LP,LP2,Wu,ZZB}), the best general upper bound is still $2(\Delta - 1)^2$.
	
	In this paper, we improve this upper bound by $\Delta-3$.
	\begin{Theorem}\label{thm:main}
		Let $G$ be a graph with maximum degree $\Delta \ge 4$. Then $\chi'_{inj}(G) \le 2(\Delta - 1)^2 - \Delta + 3$.
	\end{Theorem}

	In order to prove our main theorem, for every graph $G$, we will construct an auxiliary graph $H$ such that $\chi(H)=\chi'_{inj}(G)$.
	
	For a graph $G$ and two edges $e$ and $e'$ in $G$, we say that $e$ and $e'$ \emph{conflict} if 
	\begin{itemize}
		\item $e$ and $e'$ lie in a common triangle, or
		\item $e$ and $e'$ are at distance $2$.
	\end{itemize}
	
	Under an injective edge coloring of $G$, any two edges that conflict receive distinct colors. We define the \emph{conflict graph} $H$ of $G$ by setting $V(H) = E(G)$, with two vertices $e,e'\in V(H)$ adjacent if and only if $e$ and $e'$ conflict in $G$. By definition, $\chi(H)=\chi'_{inj}(G)$.

	\section{Proof of Theorem 1.1}
	Throughout this section, let $G$ be a graph with maximum degree $\Delta \ge 4$, and let $H$ be its conflict graph.
	\begin{Lemma}\label{lem:max-degree}
		Let $G$ be a graph with maximum degree $\Delta \ge 4$. Then $\Delta(H) \le 2(\Delta - 1)^2$.
	\end{Lemma}
	\begin{proof}
		Let $e=uv\in E(G)$. For each $x\in N_G(u)\setminus\{v\}$, let
		\[
		E_u(x)=\{f\in E(G): x\in f,\ f\neq ux\},
		\]
		and for each $y\in N_G(v)\setminus\{u\}$, let
		\[
		E_v(y)=\{f\in E(G): y\in f,\ f\neq vy\}.
		\]
		We claim that every edge conflicting with $e$ belongs to
		\[
		\Biggl(\bigcup_{x\in N_G(u)\setminus\{v\}}E_u(x)\Biggr)
		\;\bigcup\;
		\Biggl(\bigcup_{y\in N_G(v)\setminus\{u\}}E_v(y)\Biggr).
		\]
		
		Let $f$ be an edge conflicting with $e$. If $f$ is disjoint from $e$, then $e$ and $f$ are at distance $2$. Hence there is an edge joining an endpoint of $e$ to an endpoint of $f$. Thus some endpoint of $f$ lies in $N_G(u)\setminus\{v\}$ or in $N_G(v)\setminus\{u\}$, and consequently $f$ belongs to $\bigcup_{x\in N_G(u)\setminus\{v\}}E_u(x)$ or $\bigcup_{y\in N_G(v)\setminus\{u\}}E_v(y)$.
		
		If $f$ is incident with $e$, say $f=ua$, then $f$ can conflict with $e=uv$ only if $uav$ is a triangle. Hence $a\in N_G(v)\setminus\{u\}$, and $f\in E_v(a)$. The case in which $f$ is incident with $v$ is symmetric.
		
		Therefore,
		\begin{equation*}
			\begin{split}
				\deg_H(e)
				&\le
				\sum_{x\in N_G(u)\setminus\{v\}}(\deg_G(x)-1)
				+
				\sum_{y\in N_G(v)\setminus\{u\}}(\deg_G(y)-1) \\
				&\le
				2(\Delta - 1)^2.
			\end{split}
		\end{equation*}
		As $e$ was arbitrary, we have $\Delta(H)\le 2(\Delta - 1)^2$. This completes the proof.
	\end{proof}
	
	\begin{Lemma}\label{lem:local-structure}
		Let $G$ be a graph with maximum degree $\Delta$, let $e = uv \in E(G)$ satisfy $\deg_H(e) \ge 2(\Delta - 1)^2 - \Delta + 3$, and let
		\[A_e = (N_G(u)\setminus \{v\}) \cup (N_G(v)\setminus \{u\}).\]
		Then the following hold:
		\begin{enumerate}
			\item[(1)] $\deg_G(u) = $ $\deg_G(v) = \Delta$;
			\item[(2)] $e$ is not contained in a triangle in $G$;
			\item[(3)] $e_G(A_e) \le \Delta - 3$.
		\end{enumerate}
		Moreover, if $\deg_H(e) \ge 2(\Delta - 1)^2 - \Delta + 4$, then $e_G(A_e) \le \Delta - 4$.
	\end{Lemma}
	\begin{proof}
		Suppose to the contrary that $\deg_G(u) \le \Delta - 1$ or $\deg_G(v) \le \Delta - 1$. Then by the counting argument in Lemma \ref{lem:max-degree}, we have $\deg_H(e) \le (\Delta - 2)(\Delta - 1) + (\Delta - 1)^2 = 2(\Delta - 1)^2 - \Delta + 1$, a contradiction. Thus (1) holds.
		
		Suppose to the contrary that $e=uv$ lies in a triangle $uvw$. Following the proof of Lemma \ref{lem:max-degree}, we first consider $E_u(w) \cup E_v(w)$. Since $w$ is adjacent to both $u$ and $v$, it follows that $|E_u(w) \cup E_v(w)| \le \deg_G(w) \le \Delta$. So there are at most $\Delta - 2$ neighbors other than $w$ at each of $u$ and $v$, and each neighbor contributes at most $\Delta - 1$ conflicting edges with $e$. Hence
		\begin{equation*}
			\begin{split}
				\deg_H(e) 
				&\le 
				\deg_G(w) + \sum_{x\in N_G(u)\setminus\{v,w\}}(\deg_G(x)-1) + \sum_{y\in N_G(v)\setminus\{u,w\}}(\deg_G(y)-1) \\
				&\le
				\Delta + 2(\Delta-2)(\Delta - 1)\\
				&= 2(\Delta - 1)^2 - \Delta + 2, 
			\end{split}			
		\end{equation*}
		which contradicts the assumption that $\deg_H(e) \ge 2(\Delta - 1)^2 - \Delta + 3$. Thus $e$ is not contained in a triangle and (2) holds.
		
		By (1) and (2), it holds that $|A_e| = 2\Delta - 2$, and each vertex $a \in A_e$ has exactly one neighbor in $\{u,v\}$. Then the unique edge joining $a$ to $\{u,v\}$ does not conflict with $e$, while all other edges incident with $a$ conflict with $e$. Thus the sum $\sum_{a \in A_e}(\deg_G(a) - 1)$ counts all edges in $G$ conflicting with $e$. As every such edge with both endpoints in $A_e$ is counted twice,  we have 
		\[\deg_H(e) = \sum_{a \in A_e}(\deg_G(a) - 1) - e_G(A_e) \le 2(\Delta - 1)^2 - e_G(A_e).\]
		Since $\deg_H(e) \ge 2(\Delta - 1)^2 - \Delta + 3$, $e_G(A_e) \le \Delta - 3$. This proves (3).
        
		Similarly, we can prove that if $\deg_H(e) \ge 2(\Delta - 1)^2 - \Delta + 4$, then $e_G(A_e) \le \Delta - 4$.
	\end{proof}
	
	In the following, we let $D = 2(\Delta - 1)^2$. For an integer $t$, let 
	\[H_{\ge t} = H[\{v \in V(H): \deg_H(v) \ge t\}]\]
	be the subgraph of $H$ induced by the vertices with degree at least $t$. Note that, by the definition of $H$, a vertex in $H$ with degree at least $t$ corresponds to an edge in $G$ conflicting with at least $t$ other edges.
	
	\begin{Lemma}\label{lem:first-clique}
		If $\Delta(G) = \Delta$, then $\omega(H_{\ge D - \Delta + 3}) \le 2\Delta - 1$.
	\end{Lemma}
	\begin{proof}
		Let $Q$ be a clique in $H_{\ge D - \Delta + 3}$, and let $E'\subseteq E(G)$ be the set of edges in $G$ corresponding to the vertices in $Q$, so any two edges in $E'$ conflict. By Lemma~\ref{lem:local-structure}(2), no edge of $E'$ lies in a triangle. Since two adjacent edges can conflict only when they lie in a common triangle, the edges of $E'$ are pairwise disjoint.
		
		Fixing an edge $e = uv \in E'$, every edge $e' \in E' \setminus \{e\}$ conflicts with $e$ and hence there is an edge joining an endpoint of $e$ to an endpoint of $e'$. We call such an edge a \emph{connector} between $e$ and $e'$. Choose one such connector for each edge of $E' \setminus \{e\}$. Note that these connectors are distinct, because the edges of $E'$ are pairwise disjoint. Since there are at most $2\Delta - 2$ edges incident with $u$ or $v$ other than $e$, it follows that $|Q| - 1 \le 2\Delta - 2$. Hence $|Q| \le 2\Delta - 1$. This completes the proof.
	\end{proof}
	
	\begin{Lemma}\label{lem:second-clique}
		If $\Delta(G) = \Delta$, then $\omega(H_{\ge D - \Delta + 4}) \le 2\Delta - 3$.
	\end{Lemma}
	\begin{proof}
		Suppose to the contrary that $H_{\ge D - \Delta + 4}$ contains a clique $Q$ with $|Q| = 2\Delta - 2$. Let $E' = \{e_1,e_2,\ldots,e_{2\Delta-2}\} \subseteq E(G)$ be the set of edges in $G$ corresponding to the vertices in $Q$, so any $e_i$ and $e_j$ with $i,j\in [2\Delta-2]$ and $i\neq j$ conflict. By Lemma~\ref{lem:local-structure}(2), none of $e_1,e_2,\ldots,e_{2\Delta-2}$ is contained in a triangle. So these $2\Delta-2$ edges are pairwise disjoint.
		
		Let $e_i = a_ib_i$ for each $i \in [2\Delta-2]$. Since $e_i$ and $e_j$ are disjoint and conflict for any $i \neq j$, there is a connector between $e_i$ and $e_j$. For each unordered pair $\{i,j\}\subseteq[2\Delta-2]$, fix a connector $c_{ij}=c_{ji}$ between $e_i$ and $e_j$. For each $i \in [2\Delta-2]$, let 
		\[A_i = (N_G(a_i)\setminus \{b_i\}) \cup (N_G(b_i)\setminus \{a_i\}).\]
		Then by Lemma~\ref{lem:local-structure}, we have
		\begin{align}\label{Delta-4}
			e_G(A_i) \le \Delta - 4,\qquad \mbox{for each }i \in [2\Delta-2]. 
		\end{align}
		Now consider a triple $T  = \{i,j,k\} \subseteq [2\Delta-2]$. We say that $i$ is good in $T=\{i,j,k\}$ if $c_{ij}$ and $c_{ik}$ are incident with the same endpoint of $e_i$. Let $t(T)$ be the number of good indices in $T$. Fixing $i$, let $r_i$ and $s_i$ be the numbers of edges in $\{c_{ij}:j\ne i\}$ incident with $a_i$ and $b_i$, respectively. Then $r_i + s_i = (2\Delta - 2) - 1 = 2\Delta - 3$. So the number of triples in which $i$ is good is 
		\[{r_i \choose 2} + {s_i \choose 2} \ge {\Delta - 2 \choose 2} + {\Delta - 1 \choose 2 } = (\Delta - 2)^2,\]
		because the left-hand side is minimized when $r_i$ and $s_i$ differ by at most one. Summing over all $2\Delta - 2$ choices of $i$, we obtain
		\begin{align}\label{108}
			\sum_{T} t(T) \ge (2\Delta-2)(\Delta-2)^2.
		\end{align}
		
		Now we relate this to the edges of $G[A_i]$. Consider a triple $T = \{i,j,k\}$. Assume that both $j$ and $k$ are good in $T$. Since $j$ is good, $c_{jk}$ and $c_{ij}$ are incident with the same endpoint of $e_j$. This endpoint belongs to $A_i$. Similarly, since $k$ is good, the endpoint of $e_k$ incident with $c_{jk}$ also belongs to $A_i$. Hence $c_{jk}\in E(G[A_i])$.
		
		Thus, $T$ contributes at least ${t(T) \choose 2}$ to the sum $\sum_{i=1}^{2\Delta - 2} e_G(A_i)$. Therefore,
		\[\sum_{i=1}^{2\Delta - 2} e_G(A_i) \ge \sum_T {t(T) \choose 2}. \]
		Since $0 \le t(T) \le 3$, we have
		\begin{align}\label{t(T)}
			{t(T) \choose 2} \ge t(T) - 1.
		\end{align}
		There are ${2\Delta - 2 \choose 3} $ choices of the triple $T$, so by\eqref{108} and \eqref{t(T)}, we have
		\begin{equation*}
			\begin{split}
				\sum_{i=1}^{2\Delta - 2} e_G(A_i) &\ge \sum_T(t(T) - 1) \\
					&\ge (2\Delta - 2)(\Delta-2)^2 - {2\Delta - 2 \choose 3} \\
					&= \frac{(2\Delta-2)(\Delta-2)(\Delta-3)}{3}.
			\end{split}
		\end{equation*}
		By \eqref{Delta-4}, $\sum_{i=1}^{2\Delta - 2} e_G(A_i) \le (2\Delta - 2)(\Delta - 4)$. This is impossible, because 
		\[ (\Delta - 4) - \frac{(\Delta-2)(\Delta-3)}{3}  = -\frac{(\Delta-4)^2 + 2}{3} < 0,\]
		which implies that $(2\Delta - 2)(\Delta - 4) < \frac{(2\Delta-2)(\Delta-2)(\Delta-3)}{3}$, a contradiction. Therefore, $\omega(H_{\ge D - \Delta + 4}) \le 2\Delta - 3$, and the proof is complete.
	\end{proof}

	A \emph{vertex critical graph} is a graph whose chromatic number decreases upon the removal of any vertex. If a graph $G$ satisfies $\chi(G) = k$ and $\chi(G - v) = k - 1$ for every vertex $v \in V(G)$, then $G$ is called \emph{$k$-vertex critical}. Note that a $k$-vertex critical graph $G$ satisfies $\delta(G) \ge k - 1$. For a vertex critical graph $G$, we call a vertex $v \in V(G)$ \emph{low} if $\deg_G(v) = \chi(G) - 1$, and \emph{high} if $\deg_G(v) \ge \chi(G)$. Let $\mathcal{H}(G)$ denote the subgraph of $G$ induced by its high vertices.
	
	To prove Theorem \ref{thm:main}, we also need to apply a result due to Rabern~\cite{RABERN-2013}.
	\begin{Lemma}[Rabern~\cite{RABERN-2013}]\label{lem:Rabern}
		Let $G$ be a vertex-critical graph with $\chi(G) \ge \Delta(G) + 1 - p \ge 4$ for some $p \in \mathbb{N}$. If $\omega (\mathcal{H}(G)) \le \frac{\chi(G)+1}{p+1} - 2$, then $G = K_{\chi(G)}$ or $G = O_5$, where $O_5$ is the graph shown in Figure~\ref{O5}.
		\begin{figure}[H]
			\centering
			\input{O5}
		\end{figure}
	\end{Lemma}
	
	Now we are ready to prove Theorem~\ref{thm:main}.
	\begin{proof}[Proof of Theorem~\ref{thm:main}]
		Let $G$ be a graph with maximum degree $\Delta \ge 4$, and let $H$ be its conflict graph. Let $D = 2(\Delta - 1)^2$. Suppose to the contrary that $\chi(H) \ge D - \Delta + 4$. Choose an induced subgraph $H'$ of $H$ with chromatic number at least $D - \Delta + 4$, and subject to this, $|H'|$ is minimum. By the minimality of $H'$ and since removing a vertex decreases the chromatic number by at most $1$, $H'$ is $(D - \Delta + 4)$-vertex critical, and hence $\chi(H') = D - \Delta + 4$. By Lemma~\ref{lem:max-degree}, we have
		\[ D - \Delta + 3 \le \delta(H') \le \Delta(H') \le D.\]
		So every vertex of  $H'$ has degree at least $D - \Delta + 3$, and therefore $H'$ is a subgraph of $H_{\ge D - \Delta +3}$. Then by Lemma~\ref{lem:first-clique} and $\Delta \ge 4$, $\omega(H') \le \omega(H_{\ge D - \Delta +3})  \le 2\Delta - 1 < D - \Delta + 4$.
        
		Also, by definition, each high vertex of $H'$ has degree at least $D -\Delta + 4$ in $H$. Then by Lemma~\ref{lem:second-clique}, $\omega(\mathcal{H}(H')) \le \omega(H_{\ge D - \Delta + 4}) \le 2\Delta - 3$. 
		
		Let $p = \Delta - 3 \ge 1$. Then 
		\[\chi(H') = D - \Delta + 4 \ge  \Delta(H') + 1 - p \ge D - \Delta + 3 +1 - p = 2\Delta^2 - 6\Delta + 9 \ge 17 > 4. \]
		Moreover,
		\[\frac{\chi(H') + 1}{p + 1} - 2 = 2\Delta - 3 + \frac{5}{\Delta-2} > 2\Delta - 3 \ge \omega(\mathcal{H}(H')).\]
		So we can apply Lemma~\ref{lem:Rabern} to conclude that $H' = K_{D - \Delta +4}$ or $H' = O_5$. If $H' = K_{D - \Delta +4}$, it contradicts that $\omega(H') < D - \Delta + 4$; if $H' = O_5$, then $\chi(O_5) = 5$, contradicting $\chi(H') = D - \Delta + 4 \ge 17$. Thus, $\chi(H) \le D - \Delta + 3$, and therefore $\chi'_{inj}(G) = \chi(H) \le 2(\Delta - 1)^2 - \Delta + 3$. This completes the proof.
	\end{proof}

	\section*{Declaration of generative AI and AI-assisted technologies in the manuscript preparation process}
    
    During the preparation of this work, the authors used OpenAI's GPT-5.6 Sol in order to study the special cases $\Delta = 4$ and $\Delta = 5$. After using this tool/service, the authors reviewed and edited the content as needed and take full responsibility for the content of the published article.

\end{document}

%% file: O5.tex
\begin{tikzpicture}[
    baseline=-0.6ex,
    scale=0.55,
    every node/.style={
        circle,
        fill=black,
        inner sep=1.3pt
    },
    every edge/.style={line width=0.55pt}
]

% K_5 - xy
\node (a) at (-2.15, 0.95) {};
\node (b) at (-2.85, 0) {};
\node (c) at (-2.15,-0.95) {};

\node (x) at (-1.15, 0.48) {};
\node (y) at (-1.15,-0.48) {};

\draw (a)--(b)--(c)--(a);

\draw (x)--(a);
\draw (x)--(b);
\draw (x)--(c);

\draw (y)--(a);
\draw (y)--(b);
\draw (y)--(c);

% K_4
\node (p) at (-0.15, 0.95) {};
\node (q) at (0.75, 0.48) {};
\node (r) at (0.75,-0.48) {};
\node (s) at (-0.15,-0.95) {};

\draw (p)--(q);
\draw (p)--(r);
\draw (p)--(s);
\draw (q)--(r);
\draw (q)--(s);
\draw (r)--(s);

% Joining edges
\draw (x)--(p);
\draw (x)--(q);
\draw (y)--(r);
\draw (y)--(s);

\end{tikzpicture}
\caption{The graph $O_5$.}
\label{O5}